\documentclass[12pt,reqno]{amsart}
\usepackage{mathrsfs}
\usepackage[breaklinks]{hyperref}
\usepackage[dvipsnames]{xcolor} 
\usepackage{diagbox}
\usepackage{latexsym}
\usepackage{upref, eucal}
\usepackage{tikz}
\usepackage{amsmath, amssymb, graphicx}
\usepackage{epsfig}
\usepackage{tkz-euclide}
\usepackage[headheight=110pt,top=1in, bottom=.9in, left=0.8in, right=0.8in]{geometry}
\allowdisplaybreaks
\usepackage{url}
\usepackage{amssymb}

\renewcommand{\(}{\left\(}
\renewcommand{\)}{\right\)}
\renewcommand{\[}{\left\[}
\renewcommand{\]}{\right\]}
 
\newcommand{\qbinom}[2]{\genfrac{[}{]}{0pt}{}{#1}{#2}}
\numberwithin{equation}{section}
\theoremstyle{plain}
\newtheorem{theorem}{Theorem}[section]
\newtheorem{lemma}[theorem]{Lemma}
\newtheorem{corollary}[theorem]{Corollary}

\newtheorem{remark}[]{Remark}

\makeatletter
\def\proof{\@ifnextchar[{\@oproof}{\@nproof}}
\def\@oproof[#1][#2]{\trivlist\item[\hskip\labelsep\textit{#2 Proof of\
		#1.}~]\ignorespaces}
\def\@nproof{\trivlist\item[\hskip\labelsep\textit{Proof.}~]\ignorespaces}
\makeatother

\makeatletter
\def\@tocline#1#2#3#4#5#6#7{\relax
	\ifnum #1>\c@tocdepth % then omit
	\else
	\par \addpenalty\@secpenalty\addvspace{#2}%
	\begingroup \hyphenpenalty\@M
	\@ifempty{#4}{%
		\@tempdima\csname r@tocindent\number#1\endcsname\relax
	}{%
		\@tempdima#4\relax
	}%
	\parindent\z@ \leftskip#3\relax \advance\leftskip\@tempdima\relax
	\rightskip\@pnumwidth plus4em \parfillskip-\@pnumwidth
	#5\leavevmode\hskip-\@tempdima
	\ifcase #1
	\or\or \hskip 1em \or \hskip 2em \else \hskip 3em \fi%
	#6\nobreak\relax
	\dotfill\hbox to\@pnumwidth{\@tocpagenum{#7}}\par% <---- \dotfill -> \hfill
	\nobreak
	\endgroup
	\fi}
\makeatother

\allowdisplaybreaks

\begin{document}
	
	\title[Some identities of the sums-of-tails type]{Some identities of the sums-of-tails type}
	\author{Atul Dixit, Gaurav Kumar and Aviral Srivastava}
	\address{Department of Mathematics, Indian Institute of Technology Gandhinagar, Palaj, Gandhinagar 382355, Gujarat, India} 
	\email{adixit@iitgn.ac.in; kumargaurav@iitgn.ac.in; aviral.srivastava@iitgn.ac.in}
	\thanks{2020 \textit{Mathematics Subject Classification.} Primary 11P82; Secondary 11P84, 05A17.\\
		\textit{Keywords and phrases.   Sums-of-tails identities, Ramanujan's sigma function, weighted partition identities}}
	\maketitle

\begin{center}
	\dedicatory{\emph{Dedicated to Professor Krishnaswami Alladi towards the creation and efficient management of a journal devoted to mathematics influenced by Ramanujan}}
\end{center}
	\begin{abstract}
A new sums-of-tails identity involving two parameters $b$ and $d$ is obtained and is used to derive more results of similar type. One of Ramanujan's sums-of-tails identities from the Lost Notebook is shown to be a special case of our result. In the course of deriving Ramanujan's identity, we obtain a new result of combinatorial significance. Two new representations for an infinite series associated to a mock theta function are derived. Also, we give an application of an identity of Andrews and Onofri. 
	\end{abstract}

\tableofcontents	
\section{Introduction}\label{intro}
Ramanujan's Lost Notebook \cite{lnb} is a source of inspiration to studying the implications of several original concepts ensconced in it. One such concept is a pair of what are now known as the ``sums-of-tails'' identities \cite[p.~14]{lnb}. These are given by 
\begin{equation}\label{rsi1}
	\sum_{n=0}^{\infty}\left((-q;q)_{\infty}-(-q;q)_n\right)=(-q;q)_{\infty}\left(-\frac{1}{2}+\sum_{n=1}^{\infty}\frac{q^n}{1-q^n}\right)+\frac{1}{2}\sigma(q),
\end{equation}
and
\begin{equation}\label{rsi2}
	\sum_{n=0}^{\infty}\left(\frac{1}{(q;q^2)_{\infty}}-\frac{1}{(q;q^2)_{n+1}}\right)=\frac{1}{(q;q^2)_{\infty}}\left(-\frac{1}{2}+\sum_{n=1}^{\infty}\frac{q^{2n}}{1-q^{2n}}\right)+\frac{1}{2}\sigma(q),
		\end{equation}
where $\sigma(q)$ is a famous function of Ramanujan defined by
\begin{equation*}
	\sigma(q):=\sum_{n=0}^{\infty} \frac{q^{n(n+1)/2}}{(-q;q)_n}.
\end{equation*}
Here, and throughout the paper, we consider $|q|<1$ and use the standard $q$-series notation
\begin{align*}
		(A)_0 &:=(A;q)_0 =1, \qquad \\
		(A)_n &:=(A;q)_n  = (1-A)(1-Aq)\cdots(1-Aq^{n-1})\hspace{2mm}\text{for any positive integer}\hspace{1mm} n, \\
		(A)_{\infty} &:=(A;q)_{\infty}  = \lim_{n\to\infty}(A;q)_n, \qquad |q|<1.
%		(A)_{n}&:=(A)_{\infty}/(Aq^n)_{\infty}\hspace{2mm}\text{for any integer}\hspace{1mm} n.
\end{align*}
The function $\sigma(q)$ enjoys many surprising properties and is linked to many areas of Mathematics. Andrews' paper \cite{andrews1986}, which contains the first proof of \eqref{rsi1} and \eqref{rsi2}, also explains the interesting combinatorics of $\sigma(q)$. Andrews conjectured two properties of $\sigma(q)$ in \cite{andrewsmonthly86}. They were subsequently proved by Andrews, Dyson and Hickerson \cite{adh}, who also showed that $\sigma(q)$ is intimately connected with the arithmetic of $Q(\sqrt{6})$. This interesting connection with algebraic number theory enabled them to prove Andrews' conjectures. Another beautiful paper of Cohen \cite{cohen} involves a construction of a Maass waveform from $\sigma(q)$ and its companion series $\sigma^{*}(q):=2\sum_{n=1}^{\infty}(-1)^nq^{n^2}/(q;q^2)_n$. The function $\sigma(q)$ also serves as a prototypical example of Zagier's quantum modular forms \cite{zagierqmf}. 

Zagier \cite{zagiertop} found another identity of the sums-of-tails type associated with the Dedekind eta function $\eta(\tau)$ and used it to obtain values of a certain $L$-function at negative integers. Andrews, Jim\'{e}nez-Urroz and Ono \cite{AJK_Lfunctions} found two infinite families of sums-of-tails identities connected with the sums
\begin{align}\label{sot1}
\sum_{n=0}^{\infty}\left(\frac{(t)_{\infty}}{(a)_{\infty}}-\frac{(t)_n}{(a)_n}\right)
\end{align}
and
\begin{align}\label{sot2}
	\sum_{n=0}^{\infty}\left(\frac{(a)_{\infty}(b)_{\infty}}{(c)_{\infty}(q)_{\infty}}-\frac{(a)_n(b)_n}{(c)_n(q)_n}\right), 
\end{align}
and showed that they can be used to obtain values of general $L$-functions at negative integers. There is an extensive literature on sums-of tails identities and their applications. See \cite{bandix1} for an extensive bibliography on the topic as well as \cite{chan sot}, \cite{coogan}, \cite{coogan-ono}, and \cite{folsom sot} for a few further references.

In this paper we obtain a sums-of-tails identity for the series
\begin{equation}\label{sot3}
\sum_{n=0}^{\infty}\left(\frac{1}{(b)_{\infty}(d)_{\infty}}-\frac{1}{(b)_n(d)_n}\right).
\end{equation}
The motivation behind studying this sum is now given. As shown in Andrews \cite{andrews1986}, $\sigma(q)$ is the excess number of partitions into distinct parts with even rank over those with odd rank, and it occurs as the ``error term''  of the series on the left-hand side of \eqref{rsi1}, constructed using partitions into, again, distinct parts!

The question then arises - does there exist a corresponding sums-of-tails identity if we begin, instead, with partitions in which parts differ by at least $2$ (the so-called \emph{Rogers-Ramanujan partitions}) and see if the corresponding $\sigma$-type function associated to them occurs as the ``error term'' in the sums-of-tails identity wherein the series on the left-hand side is associated with the Rogers-Ramanujan partitions?

Recently, in \cite{dks2_rascoe}, we studied, in detail, some properties of $\sigma$-type function corresponding to the Rogers-Ramanujan partitions, where it was naturally denoted by $\sigma_2(q)$, that is, $\sigma_2(q)$ is the excess number of partitions into parts differing by at least $2$ with \emph{odd} rank over those with \emph{even} rank. In the same paper, it was shown that \cite[Theorem 1.1]{dks2_rascoe}
\begin{equation}\label{sigma_2(q)}
\sigma_2(q)=\sum_{n=0}^{\infty}\frac{(-1)^nq^{n^2}}{(-q)_n}.	
\end{equation}
As mentioned above, we were curious to see if there is a connection between $\sigma_2(q)$ and the series
\begin{equation}\label{sot4}
	\sum_{n=0}^{\infty}\left(\frac{1}{(q;q^5)_{\infty}(q^4;q^5)_{\infty}}-\frac{1}{(q;q^5)_n(q^4;q^5)_n}\right).
\end{equation}
While we haven't been able to figure out such a connection yet, if at all it exists, it naturally led us to consider the sum in \eqref{sot3}.
 
Observe that one cannot reduce \eqref{sot2} to \eqref{sot3}. Moreover, unlike $_1\psi_1$ summation formula and Heine's transformation which were the principal tools in the sums-of-tails identities associated with \eqref{sot1} and \eqref{sot2} respectively, there isn't any specific identity from the theory of basic hypergeometric series which is a natural choice here. Nevertheless, the following result holds.
	\begin{theorem}\label{theorem1}
	For $b,d\in\mathbb{C}$, we have\footnote{To set $b$ or $d$ (or both) equal to $1$, first multiply both sides by $(b)_\infty$ or $(d)_\infty$ (or $(b)_\infty(d)_\infty$), and then specialize them.}
	\begin{align}\label{theorem 1 eqn}
		\sum_{n=0}^{\infty}\bigg( \frac{1}{(b)_{\infty} (d)_{\infty}} - \frac{1}{(b)_n (d)_n} \bigg) &= \frac{1}{(b)_{\infty} (d)_{\infty}} \Bigg\{ \sum_{k=1}^{\infty}\frac{q^k}{1-q^k} - \sum_{n=1}^{\infty}\frac{(-d)^n q^{n(n-1)/2}}{1-q^n} - \sum_{m=1}^{\infty}\frac{(b/q)_m q^m}{1-q^m} \notag\\
		&\quad - \sum_{n=1}^{\infty}\frac{(-d)^n q^{n(n-1)/2}}{1-q^n} \sum_{m=1}^{\infty}\frac{(q^n)_m (b/q)_m q^m}{(q)_m} \Bigg\}.
	\end{align}
\end{theorem}
A symmetric version of the above identity is given in \eqref{theorem 1 eqn sym}.

While the derivation of the known identity \cite[Equation (6.7)]{AJK_Lfunctions}, namely,
\begin{equation}\label{b=d=q}
\sum_{n=0}^{\infty}\left(\frac{1}{(q)_{\infty}^2}-\frac{1}{(q)_n^2}\right)
=\frac{1}{(q)_{\infty}^2}\left(\sum_{k=1}^{\infty}\frac{q^k}{1-q^k}-\sum_{n=1}^{\infty}\frac{(-1)^nq^{n(n+1)/2}}{1-q^n}\right)		
\end{equation}
is a trivial consequence of Theorem \ref{theorem1}, upon letting $b=d=q$, deriving Ramanujan's identity \eqref{rsi2} from it is not that easy as we shall see in this paper. One can derive \eqref{rsi2} from the sums-of-tails identity for the sum in \eqref{sot2} given in \cite[Theorem 2]{AJK_Lfunctions} in an easier way than ours (see \cite[p.~404--405]{AJK_Lfunctions}) since there are three free parameters in the sum in \eqref{sot2} whereas ours in \eqref{sot3} has only two. However, our derivation of \eqref{rsi2} in Section \ref{ramanujan sot} below gives the following new result along the way. 
\begin{theorem}\label{combination}
	We have
\begin{equation}\label{combination eqn}
\left(1-\frac{1}{q}\right)\sum_{m=1}^{\infty}(q;q^2)_{m-1}\frac{q^{2m}}{1-q^{2m}}=\sum_{n=1}^{\infty}\frac{1}{(-q)_{n-1}}\frac{q^{2n}}{1-q^{2n}}-\sum_{n=1}^{\infty}\frac{q^n}{1+q^n}.
\end{equation}
\end{theorem}
The above identity has an interesting combinatorial interpretation given next.
\begin{theorem}\label{comb eqvt}
Let $n\geq1$. Let $p_1(n)$ denote the weighted count of partitions of $n$ with the largest part even, all parts less than the largest part are odd,  distinct and not equal to one less than the largest part, and the weight of such partition $\pi$ being $(-1)^{\#_o(\pi)}$, where $\#_o(\pi)$ is the number of odd parts of $\pi$.

Let $\tau_e(n)$ (resp. $\tau_o(n)$) denote the number of even (resp. odd) divisors of $n$. Let $p_2(n)$ be the weighted count of partitions of $n$ with the largest part even, no number between (and including) half the largest part and one less than the largest part can appear as a part, the weight of such partition $\pi$ being $(-1)^{\#(\pi)-\lambda(\pi)}$, where $\#(\pi)$ and $\lambda(\pi)$ denotes the number of parts of $\pi$ and the number of appearances of the largest part respectively. Then
\begin{equation}\label{partition identity}
p_1(n)-p_1(n+1)=p_2(n)+\tau_e(n)-\tau_o(n).
\end{equation}
In particular, for any $N\geq1$,
\begin{equation}\label{pi2}
p_1(N)=\sum_{n=1}^{N-1}\left(\tau_o(n)-\tau_e(n)-p_2(n)\right).
\end{equation}
\end{theorem}
For example, if $n=6$, then the partitions enumerated by $p_1(6)$ are $6$ and $2+2+2$, and hence $p_1(6)=2$. Similarly, the only partitions enumerated by $p_1(7)$ is $6+1$ so that $p_1(7)=-1$. Also, the partitions enumerated by $p_2(6)$ are $6$, $4+1+1$ and $2+2+2$ so that $p_2(6)=3$. Since $\tau_e(6)=2=\tau_o(6)$, we see that  $p_1(6)-p_1(7)=2-(-1)=3=3+2-2=p_2(6)+\tau_e(6)-\tau_o(6)$.

We note that Merca \cite[Corollary 2.1]{merca} has represented $\tau_o(n)-\tau_e(n)$ as a finite sum involving partitions with restrictions on the parity of the number of parts.  

The two free parameters $b$ and $d$ in Theorem \ref{theorem1} allow for differentiation with respect to them, thereby giving rise to new identities. For example, Corollaries \ref{diff1} and \ref{diff2}, via differentiation and then specialization of $d$, yield the following identity. 
\begin{theorem}\label{thm series from ccd}
	We have
\begin{equation}\label{series from ccd}
\frac{1}{(q)_{\infty}^2}\sum_{n=1}^{\infty}\frac{(-1)^{n-1}nq^{n(n+1)/2}}{1-q^n}=\sum_{k=1}^{\infty}\left(\sum_{n=0}^{k-1}\frac{1}{(q)_n^2}\right)\frac{q^k}{1-q^k}=\frac{1}{(q)_{\infty}^2}\sum_{k=1}^{\infty}\left(\sum_{n=0}^{k-1}\frac{(-1)^nq^{n(n-1)/2}}{(q)_n(1-q^{k-n})}\right)\frac{(q)_kq^k}{1-q^k}.
\end{equation} 
In particular, the function on the left-hand side has positive coefficients in its power series expansion.
\end{theorem}
The series on the left-hand side of \eqref{series from ccd} arises in the representation of the generating function of the number of concave compositions of $n$ obtained by Andrews \cite[Theorem 1]{ccc}. Later, Andrews, Rhoades and Zwegers \cite[Theorem 1.2]{andrews-rhoades-zwegers} showed that the following expression containing the same series, that is,
\begin{equation*}
\frac{q^{-1/8}}{(q)_{\infty}^{3}}\left(2\sum_{n=1}^{\infty}\frac{(-1)^{n-1}nq^{n(n+1)/2}}{1-q^n}-\frac{1}{4}-2\sum_{n=1}^{\infty}\frac{q^n}{(1+q^n)^2}\right)
\end{equation*}
is a mock theta function of weight $1/2$ with shadow proportional to $\eta^3(\tau)$, where $q=e^{2\pi i\tau}$. 

\section{Proof of Theorem \ref{theorem1}}\label{main theorem}
We begin with a lemma of Abel type first derived in \cite[Proposition 2.1]{AJK_Lfunctions}; see \cite[pp. 158-160]{aar2} for a proof. It will be used in the proof of Theorem \ref{theorem1}.
\begin{lemma}\label{lemma}
Suppose that $f(z)=\sum_{n=0}^{\infty}\alpha(n)z^n$ is analytic for $|z|<1$. If $\alpha$ is a complex number for which $\sum_{n=0}^{\infty}\left(\alpha-\alpha(n)\right)<\infty$ and $\lim_{n\to\infty}n\left(\alpha-\alpha(n)\right)=0$, then	
\begin{align*}\lim_{z\to1^{-}}\frac{d}{dz}(1-z)f(z)=\sum_{n=0}^{\infty}\left(\alpha-\alpha(n)\right).	
\end{align*}
\end{lemma}

\begin{proof}[Theorem \textup{\ref{theorem1}}][]
The idea is to apply Lemma \ref{lemma} with 
\begin{equation*}
\alpha=\frac{1}{(b)_{\infty}(d)_{\infty}},\hspace{3mm} \alpha(n)=\frac{1}{(b)_{n}	(d)_{n}}. 
\end{equation*}
The hypotheses of the lemma are seen to be true using a logic similar to that given in \cite[pp.~161-162]{aar2} and so we will be brief. Indeed, employing Euler's theorem \cite[p.~9, Corollary (1.3.2)]{ntsr}
\begin{equation}\label{et} \sum_{k=0}^{\infty}\frac{(-z)^kq^{k(k-1)/2}}{(q)_k}=(z)_{\infty}\hspace{8mm}(|z|<\infty)
	\end{equation}
	in the second step below, we see that for any $b, d\in\mathbb{C}$,
\begin{align*}
\sum_{n=0}^{\infty}\left|\alpha-\alpha(n)\right|&\leq\frac{1}{(|b|;|q|)_{\infty}(|d|;|q|)_{\infty}}\sum_{n=0}^{\infty}\left|1-(bq^n)_{\infty}(dq^n)_{\infty}\right|\nonumber\\
&=\frac{1}{(|b|;|q|)_{\infty}(|d|;|q|)_{\infty}}\sum_{n=0}^{\infty}\left|\sum_{r,s=0\atop{(r,s)\neq(0,0)}}^{\infty}\frac{(-b)^r(-d)^sq^{\frac{r(r-1)}{2}+\frac{s(s-1)}{2}+n(r+s)}}{(q)_r(q)_s}\right|\nonumber\\
&\leq\frac{1}{(|b|;|q|)_{\infty}(|d|;|q|)_{\infty}}\sum_{n=0}^{\infty}|q|^n\sum_{r,s=0\atop{(r,s)\neq(0,0)}}^{\infty}\frac{|b|^r|d|^s|q|^{\frac{r(r-1)}{2}+\frac{s(s-1)}{2}}}{(|q|;|q|)_r(|q|;|q|)_s}\nonumber\\
&=\frac{1}{(|b|;|q|)_{\infty}(|d|;|q|)_{\infty}(1-|q|)}\left((-|b|;|q|)_{\infty}(-|d|;|q|)_{\infty}-1\right)\nonumber\\
&<\infty.
\end{align*}
Moreover, since $r+s-1\geq0$,
\begin{align*}
\lim_{n\to\infty}n\left(\alpha-\alpha(n)\right)=\lim_{n\to\infty}\frac{-nq^n}{(b;q)_{\infty}(d;q)_{\infty}}\sum_{r,s=0\atop{(r,s)\neq(0,0)}}^{\infty}\frac{(-b)^r(-d)^sq^{\frac{r(r-1)}{2}+\frac{s(s-1)}{2}+n(r+s-1)}}{(q)_r(q)_s}=0.
\end{align*}
Therefore, applying Lemma \ref{lemma}, we see that
\begin{align}\label{after applying}
\sum_{n=0}^{\infty}\left(\frac{1}{(b)_{\infty}(d)_{\infty}}-\frac{1}{(b)_{n}	(d)_{n}}\right)=\lim_{z\to1^{-}}\frac{d}{dz}(1-z)f(z),
\end{align}
where
\begin{equation*}
f(z):=\sum_{n=0}^{\infty}\frac{z^n}{(b)_n(d)_n}.
\end{equation*}
Using Euler's theorem \eqref{et} in the second step below, we see that
\begin{align}
	f(z) &=\frac{1}{(d)_{\infty}}\sum_{m=0}^{\infty}\frac{z^m}{(b)_m}(dq^m)_{\infty}\nonumber\\
	&=\frac{1}{(d)_{\infty}}\sum_{m=0}^{\infty}\frac{z^m}{(b)_m}\sum_{n=0}^{\infty}\frac{(-dq^m)^nq^{n(n-1)/2}}{(q)_n}\nonumber\\
	&=\frac{1}{(d)_{\infty}}\sum_{n=0}^{\infty}\frac{(-d)^nq^{n(n-1)/2}}{(q)_n}\sum_{m=0}^{\infty}\frac{(zq^n)^m}{(b)_m}\nonumber\\
	&=\frac{(q)_{\infty}}{(z)_{\infty}(b)_{\infty}(d)_{\infty}}\sum_{n=0}^{\infty}\frac{(z)_n(-d)^nq^{n(n-1)/2}}{(q)_n}\sum_{m=0}^{\infty}\frac{(zq^n)_m(bq^{-1})_mq^m}{(q)_m}\label{intermediate}\\
	&=\frac{(q)_{\infty}}{(b)_{\infty}(d)_{\infty}}\bigg\{\frac{1}{(z)_\infty}+\sum_{n=1}^{\infty}\frac{(-d)^nq^{n(n-1)/2}}{(q)_n(zq^n)_{\infty}}+\sum_{m=1}^{\infty}\frac{(bq^{-1})_mq^m}{(q)_m(zq^{m})_{\infty}}\nonumber\\
	&\quad+\sum_{n=1}^{\infty}\frac{(-d)^nq^{n(n-1)/2}}{(q)_n}\sum_{m=1}^{\infty}\frac{(bq^{-1})_mq^m}{(q)_m(zq^{m+n})_{\infty}}\bigg\},\label{main argument}
\end{align}
where, in the second last step, we used Heine's transformation \cite[p.~38]{gea} with $a=0, b=q$, then $c=b$ and $t=zq^n$.

Now multiply both sides of \eqref{main argument} by $(1-z)$, differentiate both sides with respect to $z$, then let $z\to1^{-}$ and use \eqref{after applying} to get \eqref{theorem 1 eqn}, where we repeatedly use the facts that
\begin{align*}
\left.\frac{d}{dz}\frac{(1-z)}{(z)_{\infty}}\right|_{z=1}=\left.\frac{-1}{(zq)_\infty}\sum_{k=1}^{\infty}\frac{d}{dz}\log(1-zq^k)\right|_{z=1}=\frac{1}{(q)_{\infty}}\sum_{k=1}^{\infty}\frac{q^k}{1-q^k},
\end{align*}
and, for $j\geq1$,
\begin{align*}
\left.\frac{d}{dz}\frac{(1-z)}{(zq^j)_{\infty}}\right|_{z=1}=\left.\frac{d}{dz}(1-z)\sum_{k=0}^{\infty}\frac{z^kq^{jk}}{(q)_k}\right|_{z=1}=-\frac{1}{(q^j)_{\infty}}.
\end{align*}
\end{proof}
\begin{remark}
The second iterate of Heine's transformation \cite[p.~38]{gea} with $c\to 0$ and $t=q$ gives
\begin{equation*}
	\sum_{m=0}^{\infty}\frac{(a)_m(b)_mq^m}{(q)_m}=\frac{(bq)_{\infty}}{(q)_{\infty}}\sum_{m=0}^{\infty}\frac{(b)_m(-a)^mq^{m(m+1)/2}}{(bq)_m(q)_m}.
	\end{equation*}
If we let $b=zq^n$ and $a=b/q$ in the above equation, we get
\begin{align*}
\sum_{m=0}^{\infty}\frac{(zq^n)_m(bq^{-1})_mq^m}{(q)_m}=\frac{(zq^{n+1})_{\infty}}{(q)_{\infty}}\sum_{m=0}^{\infty}\frac{(zq^n)_m(-b)^mq^{m(m-1)/2}}{(q)_m(zq^{n+1})_m}.
\end{align*} 
If we substitute the right-hand side of the above equation for the sum over $m$ in \eqref{intermediate} and then follow the same procedure as in the proof of Theorem \ref{theorem1}, we arrive at its following symmetric version:
\begin{align}\label{theorem 1 eqn sym}
	\sum_{n=0}^{\infty}\bigg( \frac{1}{(b)_{\infty} (d)_{\infty}} - \frac{1}{(b)_n (d)_n} \bigg) &= \frac{-1}{(b)_{\infty} (d)_{\infty}} \Bigg\{  \sum_{n=1}^{\infty}\frac{(-d)^n q^{n(n-1)/2}}{(q)_n(1-q^n)}+\sum_{m=1}^{\infty}\frac{(-b)^m q^{m(m-1)/2}}{(q)_{m}(1-q^m)}\nonumber\\
	&\quad+\sum_{n=1}^{\infty}\frac{(-d)^n q^{n(n-1)/2}}{(q)_n}\sum_{m=1}^{\infty}\frac{(-b)^m q^{m(m-1)/2}}{(q)_{m}(1-q^{m+n})} \Bigg\}\nonumber\\
	&=\frac{1}{(b)_{\infty} (d)_{\infty}} \Bigg\{2\sum_{k=1}^{\infty}\frac{q^k}{1-q^k}-\sum_{n=1}^{\infty}\frac{(d/q)_nq^n}{1-q^n}-\sum_{m=1}^{\infty}\frac{(b/q)_mq^m}{1-q^m}\nonumber\\
	&\quad-\sum_{n=1}^{\infty}\frac{(-d)^n q^{n(n-1)/2}}{(q)_n}\sum_{m=1}^{\infty}\frac{(-b)^m q^{m(m-1)/2}}{(q)_{m}(1-q^{m+n})} \Bigg\},
\end{align}
where, in the last step, we used \cite[Theorem 1.16]{gupta}.
\end{remark}

\section{Proof of Ramanujan's identity \eqref{rsi2}}\label{ramanujan sot}
	\begin{lemma}\label{theorem3}
	\begin{align*}
		\sum_{n=1}^{\infty} \frac{q^{2n}}{(-q)_n(1-q^n)}
		&= \frac{q}{1-q}-(1-q)\sum_{m=1}^{\infty}\frac{q^m}{(-q)_m(1-q^m)(1-q^{m+1})}
	\end{align*}
\end{lemma}
\begin{proof}
		\begin{align*}
		\sum_{n=1}^{\infty}\frac{q^{2n}}{(-q)_n(1-q^n)} &=\frac{q^2}{1-q^2}+\sum_{n=2}^{\infty}\frac{q^{2n}}{(-q)_n(1-q^n)}\nonumber\\
	%	&=\frac{q}{1-q}-\frac{q}{1-q^2}+\sum_{n=2}^{\infty}\frac{q^{2n}}{(-q)_{n-1}(1-q^{2n})}\nonumber\\
	%	&=\frac{q}{1-q}-\frac{q}{1-q^2}+\sum_{m=1}^{\infty}\frac{q^{2m+2}}{(-q)_m(1-q^{2m+2})}\nonumber\\
		&=\frac{q^2}{1-q^2}+\sum_{m=1}^{\infty}\frac{q^{m+1}}{(-q)_m(1-q^{m+1})}\bigg\{1-\frac{1}{1+q^{m+1}}\bigg\}\nonumber\\
		&=\frac{q^2}{1-q^2}+\sum_{m=1}^{\infty}\frac{q^{m+1}}{(-q)_m(1-q^{m+1})}-\sum_{m=1}^{\infty}\frac{q^{m+1}}{(-q)_{m+1}(1-q^{m+1})}\nonumber\\
		&=\frac{q}{1-q}-\frac{q}{1-q^2}+\sum_{m=1}^{\infty}\frac{q^{m+1}}{(-q)_m(1-q^{m+1})}-\sum_{m=2}^{\infty}\frac{q^{m}}{(-q)_{m}(1-q^{m})}\nonumber\\
		&=\frac{q}{1-q}+\sum_{m=1}^{\infty}\frac{q^{m+1}}{(-q)_m(1-q^{m+1})}-\sum_{m=1}^{\infty}\frac{q^{m}}{(-q)_{m}(1-q^{m})}\nonumber\\
		%&=\frac{q}{1-q}+\sum_{m=1}^{\infty}\frac{q^m}{(-q)_m}\bigg\{\frac{q}{1-q^{m+1}}-\frac{1}{1-q^m}\bigg\}\nonumber\\
		&= \frac{q}{1-q}-(1-q)\sum_{m=1}^{\infty}\frac{q^m}{(-q)_m(1-q^m)(1-q^{m+1})}.
	\end{align*}
\end{proof}

\begin{proof}[Theorem \textup{\ref{combination}}][]
	From \cite[p.~29, Exercise 2]{gea} (see also \cite[p.~8, Entry 1.2.6]{aar2}) with $t=q$ that
\begin{align*}
	\sum_{n=0}^{\infty}\frac{(b)_{2n}q^{2n}}{(q^2;q^2)_n}&=\frac{(-bq)_{\infty}}{(-q)_{\infty}}\sum_{m=0}^{\infty}\frac{(b)_mq^m}{(q)_m(-bq)_m}.
\end{align*}
Separate out the $n=0$ term, subtract $1$ from both sides and divide by $(1-bq)$ to get
\begin{align}\label{ell-1}
	(1-b)\sum_{n=1}^{\infty}\frac{(bq^2)_{2n-2}q^{2n}}{(q^2;q^2)_n}&=\frac{1}{1-bq}\bigg\{-1+\frac{(-bq)_{\infty}}{(-q)_{\infty}}\sum_{m=0}^{\infty}\frac{(b)_mq^m}{(q)_m(-bq)_m}\bigg\}.
\end{align}
We now wish to let $b\to q^{-1}$ on both sides. The right-hand side  is easily seen to be of $\frac{0}{0}$ form. Hence we use L'H\^{o}pital's rule. Observe that
\begin{align}\label{ell 0}
\lim_{b\to q^{-1}} \frac{1}{1-bq}\bigg\{-1+\frac{(-bq)_{\infty}}{(-q)_{\infty}}\sum_{m=0}^{\infty}\frac{(b)_mq^m}{(q)_m(-bq)_m}\bigg\}=\frac{-1}{q}\cdot L,
\end{align}
where
\begin{align}\label{ell}
L:=	\lim_{b\to q^{-1}} \frac{d}{db}\bigg\{-1+\frac{(-bq)_{\infty}}{(-q)_{\infty}}\sum_{m=0}^{\infty}\frac{(b)_mq^m}{(q)_m(-bq)_m}\bigg\}.
\end{align}
Now
\begin{align}\label{ell1}
\frac{d}{db}\bigg\{-1+\frac{(-bq)_{\infty}}{(-q)_{\infty}}\sum_{m=0}^{\infty}\frac{(b)_mq^m}{(q)_m(-bq)_m}\bigg\}&=\frac{(-bq)_{\infty}}{(-q)_{\infty}}\Bigg[\sum_{m=0}^{\infty}\frac{(b)_mq^m}{(q)_m(-bq)_m}\sum_{n=1}^{\infty}\frac{q^n}{1+bq^n}\nonumber\\
&\quad+\sum_{m=1}^{\infty}\frac{(b)_mq^m}{(q)_m(-bq)_m}\left\{\sum_{n=0}^{m-1}\frac{-q^n}{1-bq^n}-\sum_{n=1}^{m}\frac{q^n}{1+bq^n}\right\}\Bigg].
\end{align}
Note that
\begin{align}\label{ell2}
&\sum_{m=1}^{\infty}\frac{(b)_mq^m}{(q)_m(-bq)_m}\sum_{n=0}^{m-1}\frac{-q^n}{1-bq^n}\nonumber\\
&=\sum_{n=0}^{\infty}\frac{-q^n}{1-bq^n}\sum_{m=n+1}^{\infty}\frac{(b)_mq^m}{(q)_m(-bq)_m}\nonumber\\
&=\frac{-1}{1-b}\sum_{m=1}^{\infty}\frac{(b)_mq^m}{(q)_m(-bq)_m}-\frac{q}{1-bq}\sum_{m=2}^{\infty}\frac{(b)_mq^m}{(q)_m(-bq)_m}+\sum_{n=2}^{\infty}\frac{-q^n}{1-bq^n}\sum_{m=n+1}^{\infty}\frac{(b)_mq^m}{(q)_m(-bq)_m}.
\end{align}
Observe that when we let $b\to q^{-1}$, the double sum on the right-hand side tends to zero. 
Therefore, from \eqref{ell}, \eqref{ell1} and \eqref{ell2}, we have
\begin{align}\label{ell eval}
	L=\frac{q}{2}-\frac{q}{1-q}+\sum_{n=1}^{\infty}\frac{q^n}{1+q^{n-1}}+q(1-q)\sum_{m=1}^{\infty}\frac{q^m}{(-q)_m(1-q^m)(1-q^{m+1})}.
\end{align}
From \eqref{ell 0} and \eqref{ell eval}, we see that in the limit $b\to q^{-1}$, \eqref{ell-1} becomes
\begin{align*}
	\left(1-\frac{1}{q}\right)\sum_{n=1}^{\infty}\frac{(q;q^2)_{n-1}q^{2n}}{1-q^{2n}}&=\frac{-1}{2}+\frac{1}{1-q}-\sum_{n=0}^{\infty}\frac{q^n}{1+q^n}-(1-q)\sum_{m=1}^{\infty}\frac{q^m}{(-q)_m(1-q^m)(1-q^{m+1})}.
\end{align*}
Finally, to obtain the desired result, we invoke Lemma \ref{theorem3} in the above equation and simplify.
\end{proof}

We are now ready to prove Ramanujan's sums-of-tails identity.
\begin{proof}[ \textup{\eqref{rsi2}}][]
	In Theorem \eqref{theorem1}, separate $n=0$ term on the left-hand side, replace $n$ by $n+1$, $q$ by $q^2$, and then set $d=0,b=q$ so as to have
\begin{align}\label{app of main theorem}
	\sum_{n=0}^{\infty}\bigg\{ \frac{1}{(q;q^2)_{\infty}} - \frac{1}{(q;q^2)_{n+1}} \bigg\}
	&=\frac{1}{(q;q^2)_{\infty}}\sum_{k=1}^{\infty}\frac{q^{2k}}{1-q^{2k}}+1-\frac{1}{(q;q^2)_{\infty}} \bigg\{1+\sum_{m=1}^{\infty}\frac{(1/q;q^2)_mq^{2m}}{1-q^{2m}} \bigg\}\nonumber\\
	&=\frac{1}{(q;q^2)_{\infty}}\bigg\{\frac{-1}{2}+\sum_{k=1}^{\infty}\frac{q^{2k}}{1-q^{2k}}\bigg\}+1-\frac{1}{(q;q^2)_{\infty}}\nonumber\\&+\frac{1}{(q;q^2)_{\infty}} \bigg\{\sum_{k=0}^{\infty}\frac{q^k}{1+q^k}-\sum_{n=1}^{\infty} \frac{q^{2n}}{(-q)_n(1-q^n)} \bigg\},
\end{align}
where, in the last step, we used Theorem \ref{combination}. 

From \cite[p.~378]{Wptf} and Euler's theorem $(-q;q)_{\infty}=1/(q;q^2)_{\infty}$, we see that
\begin{align*}
	\frac{1}{(q;q^2)_{\infty}}-1=\frac{1}{(q;q^2)_{\infty}}\sum_{n=1}^{\infty}\frac{q^n}{(-q)_n}.
\end{align*}
Substituting the above equation in \eqref{app of main theorem}, we get
\begin{align}\label{final}
	\sum_{n=0}^{\infty}\bigg\{ \frac{1}{(q;q^2)_{\infty}} - \frac{1}{(q;q^2)_{n+1}} \bigg\}
	&=\frac{1}{(q;q^2)_{\infty}}\bigg\{\frac{-1}{2}+\sum_{k=1}^{\infty}\frac{q^{2k}}{1-q^{2k}}\bigg\}+\frac{1}{(q;q^2)_{\infty}} \bigg\{\sum_{k=0}^{\infty}\frac{q^k}{1+q^k}\nonumber\\&\quad-\sum_{n=1}^{\infty} \frac{q^n}{(-q)_n(1-q^n)} \bigg\}.
\end{align}
Now use the elementary fact $1/(-q)_n=(q;q^2)_n/(q^{n+1})_n$ to write
\begin{align*}
\frac{1}{(q;q^2)_{\infty}}\sum_{n=1}^{\infty} \frac{q^n}{(-q)_n(1-q^n)}&=\sum_{n=1}^{\infty} \frac{q^n}{(q^n)_{n+1}(q^{2n+1};q^2)_{\infty}}\nonumber\\
&=-\frac{1}{2}\sigma(q)+\frac{1}{(q;q^2)_{\infty}}\sum_{n=0}^{\infty}\frac{q^n}{1+q^n},
\end{align*}
where the last step follows from \cite[p.~8, Theorem 3.3]{ADY}. Finally substituting the above equation in \eqref{final} leads us to \eqref{rsi2}, thereby completing the proof.
\end{proof}

\section{A combinatorial equivalent of Theorem \ref{combination} and its proof}
In this section, we prove Theorem \ref{comb eqvt}.
Observe that
\begin{align*}
\sum_{m=1}^{\infty}(q;q^2)_{m-1}\frac{q^{2m}}{1-q^{2m}}=\sum_{m=1}^{\infty}(1-q)(1-q^3)\cdots(1-q^{2m-3})\frac{q^{2m}}{1-q^{2m}}.
\end{align*}
Thus, in a typical partition generated by the above sum, the largest part must be even, in this case $2m$, and is allowed to repeat. However, any other part less than $2m$ must be odd, distinct and strictly less than $2m-1$. Moreover, whenever an odd part appears, it comes with the weight $(-1)$. Thus, we get a partition $\pi$ weighted by $(-1)^{\#_o(\pi)}$, where $\#_o(\pi)$ is the number of odd parts of $\pi$, which is clearly enumerated by $p_1(n)$. Thus,
\begin{align}\label{one}
\left(1-\frac{1}{q}\right)	\sum_{m=1}^{\infty}(q;q^2)_{m-1}\frac{q^{2m}}{1-q^{2m}}=\sum_{n=1}^{\infty}p_1(n)q^{n}-\sum_{n=1}^{\infty}p_1(n)q^{n-1}=\sum_{n=1}^{\infty}\left(p_1(n)-p_1(n+1)\right)q^{n},
	\end{align}
since $p_1(1)=0$. Next, consider 
\begin{equation}\label{two}
\sum_{n=1}^{\infty}\frac{1}{(-q)_{n-1}}\frac{q^{2n}}{1-q^{2n}}=\sum_{n=1}^{\infty}\frac{1}{(1+q)\cdots(1+q^{n-1})}\frac{q^{2n}}{1-q^{2n}}.
\end{equation}
This is the generating function of $p_2(n)$ because, a typical partition has $2n$ as its largest part, no number between (and including) $n$ and $2n-1$ can appear as a part, and every occurrence of a part less than $n$ comes with the weight $-1$. 
Finally, 
\begin{align}\label{three}
\sum_{n=1}^{\infty}\frac{q^n}{1+q^n}=\sum_{n=1}^{\infty}\frac{q^n}{1-q^{2n}}-\sum_{n=1}^{\infty}\frac{q^{2n}}{1-q^{2n}}=\sum_{n=1}^{\infty}\left(\tau_o(n)-\tau_e(n)\right)q^n.
\end{align}
From \eqref{one}, \eqref{two}, \eqref{three} and \eqref{combination eqn}, we arrive at \eqref{partition identity}.

Equation \eqref{pi2} follows by summing both sides of \eqref{partition identity} from $n=1$ to $N-1$ and making use of the fact that $p_1(1)=0$.

\section{Further sums-of-tails identities}

	\begin{theorem}
		For $b, d\in\mathbb{C}$,
	\begin{align}\label{sot5}
		\sum_{n=0}^{\infty} \left((bq^n)_{\infty}(dq^n)_{\infty}-(q^{n+1})_{\infty}^2\right) = \sum_{j=1}^{\infty} \frac{q^j}{1-q^j} \sum_{n=0}^{j} \qbinom{j}{n} \left(\frac{b}{q}\right)_{j-n} \left(\frac{-d}{q}\right)^n q^{n(n-1)/2} - \sum_{j=1}^{\infty} \frac{(-1)^jq^{j(j+1)/2}}{(1-q^j)}.
	\end{align}
\end{theorem}
\begin{proof}
From Theorem \ref{theorem1},
\begin{align}\label{fromSOT}
	&\sum_{n=0}^{\infty}\bigg\{ \frac{1}{(b)_{\infty} (d)_{\infty}} - \frac{1}{(b)_n (d)_n} \bigg\} \nonumber\\
	&= \frac{1}{(b)_{\infty} (d)_{\infty}} \Bigg\{ \sum_{k=1}^{\infty}\frac{q^k}{1-q^k} - \sum_{m=1}^{\infty}\frac{(b/q)_m q^m}{1-q^m}
	 - \sum_{n=1}^{\infty}\frac{(-d/q)^n q^{n(n-1)/2}}{(q)_n} \sum_{m=0}^{\infty}\frac{(q)_{m+n} (b/q)_m q^{m+n}}{(q)_m(1-q^{m+n})} \Bigg\}\nonumber\\
&= \frac{1}{(b)_{\infty} (d)_{\infty}} \Bigg\{ \sum_{k=1}^{\infty}\frac{q^k}{1-q^k} - \sum_{m=1}^{\infty}\frac{(b/q)_m q^m}{1-q^m} 
	- \sum_{j=1}^{\infty} \frac{q^j}{1-q^j} \sum_{n=1}^{\infty} \qbinom{j}{n} \left(\frac{b}{q}\right)_{j-n} \left(\frac{-d}{q}\right)^n q^{n(n-1)/2} \Bigg\}\nonumber\\ 
	&= \frac{1}{(b)_{\infty} (d)_{\infty}} \Bigg\{ \sum_{k=1}^{\infty}\frac{q^k}{1-q^k} - \sum_{j=1}^{\infty} \frac{q^j}{1-q^j} \sum_{n=0}^{\infty} \qbinom{j}{n} \left(\frac{b}{q}\right)_{j-n} \left(\frac{-d}{q}\right)^n q^{n(n-1)/2} \Bigg\}.
\end{align}
Multiply both sides of \eqref{b=d=q} by $\frac{(q)_{\infty}^2}{(b)_{\infty}(d)_{\infty}}$ to get
\begin{align}\label{fromAndrews}
	\sum_{n=0}^{\infty}\bigg\{ \frac{1}{(b)_{\infty} (d)_{\infty}} - \frac{(q^{n+1})_{\infty}^2}{(b)_{\infty} (d)_{\infty}} \bigg\} = \frac{1}{(b)_{\infty} (d)_{\infty}} \bigg\{ \sum_{k=1}^{\infty}\frac{q^k}{1-q^k} - \sum_{n=1}^{\infty} \frac{(-1)^nq^{n(n+1)/2}}{1-q^n} \bigg\}.
\end{align}
Equation \eqref{sot5} now follows from subtracting (\ref{fromSOT}) from (\ref{fromAndrews}), and then multiplying both sides by $(b)_\infty(d)_\infty$.
\end{proof}
\begin{corollary}\label{diff1}
	For $b\in\mathbb{C}$,
	\begin{align*}
		\sum_{n=0}^{\infty} \left((bq^n)_{\infty}^2-(q^{n+1})_{\infty}^2\right) = \sum_{j=1}^{\infty} \frac{q^j}{1-q^j} \sum_{n=0}^{j} \qbinom{j}{n} \left(\frac{b}{q}\right)_{j-n} \left(\frac{-b}{q}\right)^n q^{n(n-1)/2} - \sum_{j=1}^{\infty} \frac{(-1)^jq^{j(j+1)/2}}{(1-q^j)}.
	\end{align*}
\end{corollary}
\begin{proof}
Set $d=b$ in \eqref{sot5}.
\end{proof}
\begin{corollary}\label{diff2}
	For $d\in\mathbb{C}$,
	\begin{align}\label{diff2 eqn}
		\sum_{n=0}^{\infty} (q^{n+1})_{\infty} \left((dq^n)_{\infty}-(q^{n+1})_{\infty}\right) = \sum_{j=1}^{\infty} \frac{(-1)^j(d^j-q^j)q^{j(j-1)/2}}{(1-q^j)}.
	\end{align}
\end{corollary}
\begin{proof}
Set $b=q$ in \eqref{sot5} and simplify.
\end{proof}
\begin{corollary}
	For $d\in\mathbb{C}$,
	\begin{align*}
\sum_{j=1}^{\infty} \frac{q^j}{1-q^j} \sum_{n=0}^{j} (-1)^n\qbinom{j}{n} \left(\frac{d}{q}\right)_{j-n}  q^{n(n-1)/2}=\sum_{j=1}^{\infty} \frac{(-d)^jq^{j(j-1)/2}}{(1-q^j)}.
\end{align*}
\end{corollary}
\begin{proof}
Let $d=q$ in \eqref{sot5}. Then replacing $b$ by $d$ in the resulting identity, we get
\begin{align}\label{diff3 eqn}
	\sum_{n=0}^{\infty} (q^{n+1})_{\infty} \left((dq^n)_{\infty}-(q^{n+1})_{\infty}\right) =\sum_{j=1}^{\infty} \frac{q^j}{1-q^j} \sum_{n=0}^{j} (-1)^n\qbinom{j}{n} \left(\frac{d}{q}\right)_{j-n}  q^{n(n-1)/2}- \sum_{j=1}^{\infty} \frac{(-1)^jq^{j(j+1)/2}}{(1-q^j)}.
\end{align}
The result now follows by equating the right-hand sides of \eqref{diff2 eqn} and \eqref{diff3 eqn}.
\end{proof}
We now prove \eqref{series from ccd}.

\begin{proof}[Theorem \textup{\ref{thm series from ccd}}][]
We first differentiate both sides of Corollary \ref{diff1} with respect to $b$ and then let $b=q$. Observe that
\begin{align*}
\left.\frac{d}{db}(bq^n)_{\infty}^2\right|_{b=q}=-2(q^{n+1})_{\infty}^2\sum_{k=0}^{\infty}\frac{q^{n+k}}{1-q^{n+k+1}}.
\end{align*}
Also, for $n\neq j$,
\begin{align*}
\left.\frac{d}{db}\left(\frac{b}{q}\right)_{j-n}b^{n}\right|_{b=q}&=\left(\frac{b}{q}\right)_{j-n}\left\{nb^{n-1}-b^n\sum_{k=0}^{j-n-1}\frac{q^{k-1}}{1-bq^{k-1}}\right\}_{b=q}\nonumber\\
&=\left.-b^n\left(\frac{b}{q}\right)_{j-n}\frac{q^{-1}}{1-bq^{-1}}\right|_{b=q}\nonumber\\
&=-q^{n-1}(q)_{j-n-1}.
\end{align*}
Now separating out the $n=j$ term in the first sum on the right-hand side of Corollary \ref{diff1}, then differentiating both sides of the resulting identity with respect to $b$ followed by letting $b=q$, and using the above two results, yields
 \begin{align*}
 -2\sum_{n=0}^{\infty}(q^{n+1})_{\infty}^2\sum_{k=0}^{\infty}\frac{q^{n+k}}{1-q^{n+k+1}}=\sum_{j=1}^{\infty}\frac{(-1)^jjq^{-1+j(j+1)/2}}{1-q^j}-\sum_{j=1}^{\infty} \frac{q^j}{1-q^j} \sum_{n=0}^{j-1}\qbinom{j}{n} (-1)^nq^{\frac{n(n-1)}{2}-1}(q)_{j-n-1}
 \end{align*}
 Multiplying both sides by $q$ gives, upon simplification,
 \begin{align}\label{first eq}
 \sum_{j=1}^{\infty}\frac{(-1)^{j-1}jq^{j(j+1)/2}}{1-q^j}=2(q)_{\infty}^2\sum_{k=1}^{\infty}\left(\sum_{n=0}^{k-1}\frac{1}{(q)_n^2}\right)\frac{q^k}{1-q^k}-\sum_{j=1}^{\infty}\left(\sum_{n=0}^{j-1}\frac{(-1)^nq^{n(n-1)/2}}{(q)_n(1-q^{j-n})}\right)\frac{(q)_jq^j}{1-q^j}.
 \end{align} 
 Now differentiating both sides of Corollary \ref{diff2} with respect to $d$, letting $d=q$, and then multiplying both sides by $q$ results in
 \begin{align}\label{second eq}
 \sum_{j=1}^{\infty}\frac{(-1)^{j-1}jq^{j(j+1)/2}}{1-q^j}=(q)_{\infty}^2\sum_{k=1}^{\infty}\left(\sum_{n=0}^{k-1}\frac{1}{(q)_n^2}\right)\frac{q^k}{1-q^k}.
 \end{align}
 Equation \eqref{series from ccd} now follows from \eqref{first eq} and \eqref{second eq}.
\end{proof}	

\section{Some results on finite sums, an identity of Andrews and Onofri, and its application}
In a beautiful paper, Andrews and Onofri gave the following identity \cite[p.~182, Entry (7.2)]{lattice}\footnote{The factor of $q^m$ in the summand of the double sum is missing in \cite{lattice}.}
\begin{align}\label{a-o identity}
	\sum_{n=0}^{\infty}		\sum_{m=0}^{\infty}	\frac{(n-m)(b/a)_n(b/a)_ma^{n+m}q^m}{(q)_n(q)_m}=\frac{(a-b)(bq)_{\infty}^2}{(a)_{\infty}^2}.
\end{align}
It is valid for $|a|\leq|b|$ and $|q|<1$. In this section, we begin with a symmetric identity, namely \eqref{symmetry eqn}, and then use a corollary of \eqref{a-o identity} to derive an identity for a special case of the finite sum $b=d$ of \eqref{symmetry eqn}. It is, in turn, used to obtain an identity whose combinatorial proof is given too. 
\begin{theorem}\label{symmetry}
	For $b, d\in\mathbb{C}$,
	\begin{align}\label{symmetry eqn}
		\sum_{n=0}^{j} \qbinom{j}{n} \left(\frac{b}{q}\right)_{j-n} \left(\frac{-d}{q}\right)^n q^{n(n-1)/2}=\sum_{n=0}^{j} \qbinom{j}{n} \left(\frac{d}{q}\right)_{j-n} \left(\frac{-b}{q}\right)^nq^{n(n-1)/2}.
	\end{align}
\end{theorem}
	\begin{proof}
	Let 
	\begin{equation*}
		T(j):=T(j, q):=\sum_{n=0}^{j} \qbinom{j}{n} \left(\frac{b}{q}\right)_{j-n} \left(\frac{-d}{q}\right)^n q^{n(n-1)/2}
	\end{equation*}
	Then
	\begin{align}\label{TJprod1}
		\sum_{j=0}^{\infty} \frac{T(j) z^j}{(q)_j} &= \sum_{j=0}^{\infty} \sum_{n=0}^{j} \qbinom{j}{n} \left(\frac{b}{q}\right)_{j-n} \left(\frac{-d}{q}\right)^n q^{n(n-1)/2} \frac{z^j}{(q)_j}\nonumber\\ &= \sum_{n=0}^{\infty} \frac{(-d/q)^nq^{n(n-1)/2}}{(q)_n} \sum_{j=n}^{\infty} \frac{(b/q)_{j-n}z^j}{(q)_{j-n}}\nonumber\\ &= \sum_{n=0}^{\infty} \frac{(-dz/q)^nq^{n(n-1)/2}}{(q)_n} \sum_{j=0}^{\infty} \frac{(b/q)_{j}z^j}{(q)_{j}}\nonumber\\ &= \frac{(dz/q)_{\infty}(bz/q)_{\infty}}{(z)_{\infty}},
	\end{align}
	where the last step follows from two special cases of the $q$-binomial theorem \cite[p.~9, Corollary (1.3.2)]{ntsr}. Since the $q$-product on the right-hand side of \eqref{TJprod1} is symmetric in $b$ and $d$, the result follows.  
\end{proof}
\begin{theorem}\label{andrews-onofri}
	For $b, d\in\mathbb{C}$,
	\begin{align}\label{theorem 5 eqn}
		\sum_{n=0}^{j} \qbinom{j}{n} \left(\frac{d}{q}\right)_{j-n} \left(\frac{-d}{q}\right)^n &q^{n(n-1)/2}\nonumber\\=\frac{(q)_j}{(1-d/q^2)} &\sum_{t=0}^{j+1} \sum_{n=0}^{t} \frac{(2n-t)(d/q^2)_n(d/q^2)_{t-n}(-1)^{j+1-t}q^{(t-n)+(j-t)(j-t+1)/2}}{(q)_n(q)_{t-n}(q)_{j+1-t}}.
	\end{align}
\end{theorem}
\begin{proof}
Letting $b=d$ in \eqref{TJprod1}, we get
\begin{align}\label{lhs}
\frac{(dz/q)_{\infty}^2}{(z)_{\infty}}=\sum_{j=0}^{\infty}\bigg(\sum_{n=0}^{j} \qbinom{j}{n} \left(\frac{d}{q}\right)_{j-n} \left(\frac{-d}{q}\right)^n &q^{n(n-1)/2}\bigg)\frac{z^j}{(q)_j}.
\end{align}
Now \eqref{a-o identity} with $a=z$ and $b=dz/q^2$ gives
\begin{align}\label{TJprod2}
	\frac{(dz/q)_{\infty}^2}{(z)_{\infty}} &= \frac{(z)_{\infty}}{z(1-d/q^2)} \sum_{n=0}^{\infty} \sum_{m=0}^{\infty} \frac{(n-m)(d/q^2)_n(d/q^2)_mq^mz^{n+m}}{(q)_n(q)_m}\nonumber\\ &= \frac{1}{(1-d/q^2)} \sum_{n=0}^{\infty} \sum_{m=0}^{\infty} \sum_{k=0}^{\infty} \frac{(n-m)(d/q^2)_n(d/q^2)_m(-1)^kq^{m+k(k-1)/2}z^{n+m+k-1}}{(q)_n(q)_m(q)_k}\nonumber\\ &= \frac{1}{(1-d/q^2)} \sum_{j=0}^{\infty} \sum_{t=0}^{j+1} \sum_{n=0}^{t} \frac{(2n-t)(d/q^2)_n(d/q^2)_{t-n}(-1)^{j+1-t}q^{(t-n)+(j-t)(j-t+1)/2}z^j}{(q)_n(q)_{t-n}(q)_{j+1-t}},
\end{align}
where we used \eqref{et} to obtain the second step, and in the last step made the substitution $m+n+k-1=j$ and $m+n=t$. On comparing the coefficients of $z^j$ in (\ref{lhs}) and (\ref{TJprod2}), we get \eqref{theorem 5 eqn}.
\end{proof}
\begin{corollary}
	For any $j\geq0$,
	\begin{equation}\label{f12}
		\sum_{n=0}^{j}\frac{(-1)^nq^{n(n+1)/2}}{(q)_n}=\sum_{n=0}^{j}\frac{(j+1-n)(-1)^nq^{n(n-1)/2}}{(q)_n}=\frac{(j+1)(-1)^jq^{j(j+1)/2}}{(q)_j}-\sum_{n=0}^{j}\frac{n(-1)^nq^{n(n-1)/2}}{(q)_n}.
	\end{equation}
\end{corollary}
\begin{proof}
The first identity follows upon letting $d\to q^2$ in Theorem \ref{andrews-onofri}. Indeed, upon taking this limit on both sides of \eqref{theorem 5 eqn} and observing that only the terms corresponding to $n=0$ and $n=t$ on the right-hand side survive, we obtain
\begin{align*}
(q)_j	\sum_{n=0}^{j}\frac{(-1)^nq^{n(n+1)/2}}{(q)_n}&=(q)_j\bigg\{\sum_{t=1}^{j+1}\frac{(-t)(q)_{t-1}(-1)^{j+1-t}q^{t+(j-t)(j-t+1)/2}}{(q)_{t}(q)_{j+1-t}}\nonumber\\
&\qquad\qquad\qquad+\sum_{t=1}^{j+1}\frac{t(q)_{t-1}(-1)^{j+1-t}q^{(j-t)(j-t+1)/2}}{(q)_{t}(q)_{j+1-t}}\bigg\}\nonumber\\
&=(q)_j\sum_{t=1}^{j+1}\frac{t(-1)^{j+1-t}q^{(j-t)(j-t+1)/2}}{(q)_{j+1-t}}\nonumber\\
&=(q)_j\sum_{t=0}^{j}\frac{(j+1-t)(-1)^tq^{t(t-1)/2}}{(q)_t},
\end{align*}
where in the last step, we replaced $t$ by $j+1-t$.

We now give a combinatorial proof of the same result which will naturally lead us to  the second equality as well.

We begin with the left-hand side of the first equality. Let $a_{e}(m, j)$ (resp. $a_o(m, j)$) denote the number of partitions of $m$ into distinct parts where the number of parts is even (resp. odd) and $\leq j$. Then
\begin{align*}
	\sum_{n=0}^{j}\frac{(-1)^nq^{n(n+1)/2}}{(q)_n}=\sum_{m=0}^{\infty}(a_e(m, j)-a_o(m, j))q^m.
\end{align*}
Next, we write
\begin{align}\label{f1}
\sum_{n=0}^{j}\frac{(j+1-n)(-1)^nq^{n(n-1)/2}}{(q)_n}=(j+1)\sum_{n=0}^{j}\frac{(-1)^nq^{n(n-1)/2}}{(q)_n}-\sum_{n=0}^{j}\frac{n(-1)^nq^{n(n-1)/2}}{(q)_n}.
\end{align}
Now $\frac{q^{n(n-1)/2}}{(q)_n}$ generates partitions into exactly $n$ or $n-1$ distinct parts. Therefore, the partitions into distinct parts exactly $n$ in number can come only from $\frac{q^{n(n+1)/2}}{(q)_{n+1}}$ and $\frac{q^{n(n-1)/2}}{(q)_n}$. The contribution in $\frac{(-1)^{n+1}q^{n(n+1)/2}}{(q)_{n+1}}+\frac{(-1)^{n}q^{n(n-1)/2}}{(q)_n}$
coming from the partitions into exactly $n$ number of parts thus gets nullified, because of opposite parity, for every $0\leq n\leq j-1$.  Thus, when we sum these terms from $n=0$ to $j$, we are left with only the partitions into exactly $j$ distinct parts with weight $(-1)^j$.  However, the latter are clearly generated by $\frac{(-1)^{j}q^{j(j+1)/2}}{(q)_j}$. We have thus proved
\begin{align}\label{f2}
\sum_{n=0}^{j}\frac{(-1)^nq^{n(n-1)/2}}{(q)_n}=\frac{(-1)^jq^{j(j+1)/2}}{(q)_j}.
\end{align}
Thus, from \eqref{f1} and \eqref{f2}, we obtain the second equality in \eqref{f12}. To complete the combinatorial proof of the first equality of \eqref{f12}, we now  show that
	\begin{equation}\label{f123}
	\frac{(j+1)(-1)^jq^{j(j+1)/2}}{(q)_j}-\sum_{n=0}^{j}\frac{n(-1)^nq^{n(n-1)/2}}{(q)_n}=\sum_{n=0}^{j}\frac{(-1)^nq^{n(n+1)/2}}{(q)_n}.
\end{equation}
We first analyze $\sum_{n=0}^{j}\frac{n(-1)^nq^{n(n-1)/2}}{(q)_n}$. The partitions into distinct parts, exactly $n$ in number, come from $\frac{(n+1)(-1)^{n+1}q^{n(n+1)/2}}{(q)_{n+1}}$ and $\frac{n(-1)^nq^{n(n-1)/2}}{(q)_n}$ with weights $(n+1)(-1)^{n+1}$ and $n(-1)^{n}$ respectively.  Thus, for every $0\leq n\leq j-1$, the partitions into exactly $n$ number of parts coming from  $\frac{(n+1)(-1)^{n+1}q^{n(n+1)/2}}{(q)_{n+1}}+\frac{n(-1)^{n}q^{n(n-1)/2}}{(q)_n}$ will have weight $(-1)^{n+1}$. Now the partitions into exactly $j$ distinct parts coming from $\frac{j(-1)^jq^{j(j-1)/2}}{(q)_j}$ will have weight $j(-1)^j$. 

Therefore, the partitions into exactly $j$ distinct parts resulting from $\frac{(j+1)(-1)^jq^{j(j+1)/2}}{(q)_j}-\frac{j(-1)^jq^{j(j-1)/2}}{(q)_j}$ will have weight $(-1)^j$. Therefore, the above analysis clearly gives \eqref{f123}. 
\end{proof}
\begin{remark}
The equality between the first and the third terms in \eqref{f12} can be proved by induction on $j$. Same is the case with \eqref{f2}.
\end{remark}

\section{Concluding Remarks}\label{cr}
It would be nice to see if the original question which led us to study the sums-of-tails identity in Theorem \ref{theorem1} can be investigated, namely, if there exists a sums-of-tails identity associated with
\begin{equation*}
	\sum_{n=0}^{\infty}\left(\frac{1}{(q;q^5)_{\infty}(q^4;q^5)_{\infty}}-\frac{1}{(q;q^5)_n(q^4;q^5)_n}\right),
\end{equation*}
and whether or not it involves $\sigma_2(q)$ defined in \eqref{sigma_2(q)}. 

Observe that Andrews' proof of \eqref{rsi1} and \eqref{rsi2} begins by obtaining alternate representations for their left-hand sides, namely \cite[Equations (2.3), (2.4)]{andrews1986},
\begin{align*}
\sum_{n=0}^{\infty}\left((-q;q)_{\infty}-(-q;q)_n\right)&=\sum_{k=1}^{\infty}kq^k(-q)_{k-1},\\
\sum_{n=0}^{\infty}\left(\frac{1}{(q;q^2)_{\infty}}-\frac{1}{(q;q^2)_{n+1}}\right)&=\sum_{k=0}^{\infty}\frac{kq^{2k+1}}{(q;q^2)_{k+1}}.
\end{align*}
We have derived an identity similar to the above for the sum in \eqref{sot4}, namely,
\begin{align}\label{sot another rep}
\sum_{n=0}^{\infty}\left(\frac{1}{(q;q^5)_{\infty}(q^4;q^5)_{\infty}}-\frac{1}{(q;q^5)_n(q^4;q^5)_n}\right)=\sum_{k=1}^{\infty}\frac{kq^{5k-4}}{(q;q^5)_k(q^4;q^5)_{k-1}}+\sum_{k=1}^{\infty}\frac{kq^{5k-1}}{(q;q^5)_k(q^4;q^5)_{k}}.
\end{align}
Indeed, this identity can be proved as follows.  Observe that any partition enumerated by $\frac{1}{(q;q^5)_{\infty}(q^4;q^5)_{\infty}}$ has to have its largest part either of the form $5k-4$ or $5k-1$, where $k\geq1$. Any partition with largest part $5k-4$ is counted $k$ times on the left-hand side of \eqref{sot another rep}, namely, once by each of its first $k$ terms, and is not counted by the rest. The logic is similar when the largest part of a partition is of the form $5k-1$. This establishes \eqref{sot another rep}. It would be nice to see if it is feasible to proceed from it towards obtaining a sums-of-tails identity similar to \eqref{rsi1} and \eqref{rsi2}.
	  
Andrews \cite{andrews1986} proved \eqref{rsi1} and \eqref{rsi2} using Ramanujan's reciprocity theorem, namely, for $|a|<1, |b|<1$,
\begin{align*}
	\rho(a, b)-\rho(b, a )=\left(\frac{1}{b}-\frac{1}{a}\right)\frac{(aq/b)_{\infty}(bq/a)_{\infty}(q)_{\infty}}{(-aq)_{\infty}(-bq)_{\infty}},
\end{align*}
where 
\begin{align*}
	\rho(a, b):=\left(1+\frac{1}{b}\right)\sum_{n=0}^{\infty}\frac{(-1)^nq^{n(n+1)/2}a^nb^{-n}}{(-aq)_{n}}.
\end{align*}
This raises a natural question - does there exist a corresponding reciprocity theorem associated with $\sigma_2(q)$? 
%This is because it is not possible to get $\sigma_2(q)$ from \eqref{2varrt}.
 Indeed, if it exists, it may throw light on the previous question. 
 
 \begin{center}
 \textbf{Acknowledgements}
 \end{center}
 The first author is supported by the Swarnajayanti Fellowship grant SB/SJF/2021-22/08 of ANRF (Government of India). The third author is supported by NBHM-DAE Fellowship (Government of India). Both the authors sincerely thank the respective funding agencies for their support.


\begin{thebibliography}{00}
	
		\bibitem{gea}
	G.~E.~Andrews, \emph{The Theory of Partitions}, Addison-Wesley Pub. Co., NY, 300 pp. (1976). Reissued, Cambridge University Press, New York, 1998.
	

	\bibitem{andrews1986}
G.~E.~Andrews, \emph{Ramanujan's ``Lost'' Notebook V: Euler's partition identity}, Adv.~Math.~\textbf{61} (1986), 156--164.

\bibitem{andrewsmonthly86}
G.~E.~Andrews, \emph{Questions and conjectures in partition theory}, Amer. Math. Monthly~\textbf{93} (1986), 708--711.

%\bibitem{andrews illinois}
%G.~E.~Andrews, \emph{Bailey chains and generalized Lambert series I - Four identities of Ramanujan}, Illinois J.~Math.~\textbf{36} no. 2 (1992), 251--274.

\bibitem{ccc}
G.~E.~Andrews, \emph{Concave and convex compositions}, Ramanujan J.~\textbf{31} no. 1-2 (2013), 67--82.

	\bibitem{aar2}
G.~E.~Andrews and B.~C.~Berndt, \emph{Ramanujan's Lost Notebook, Part \textup{II}}, Springer, New York, 2009.

	\bibitem{ADY}
G.~E.~Andrews, A.~Dixit and A.~J. ~Yee, \emph{Partitions associated with the Ramanujan/Watson mock theta functions $\omega(q), \nu(q)$ and $\phi(q)$}, Res.~Number Theory~\textbf{1:19} (2015) (25 pages). 

\bibitem{adh}
G.~E.~Andrews, F.~J.~Dyson and D.~Hickerson, \emph{Partitions and indefinite quadratic forms}, Invent.~Math.~\textbf{91} (1988), 391--407.

\bibitem{AJK_Lfunctions}
G.~E.~Andrews, J.~Jim\'{e}nez-Urroz and K.~Ono, \emph{$q$-series identities and values of certain $L$-functions}, Duke~Math.~J.~\textbf{108} No. 3 (2001), 395--419.

\bibitem{lattice}
G.~E.~Andrews and E.~Onofri, \emph{Lattice gauge theory, orthogonal polynomials and $q$-hypergeometric series.} Special Functions: Group Theoretical Aspects and Applications, R.~A.~Askey et. al. eds., Reidel Pub. Co., Boston, pp.~163--188 (1984).

\bibitem{andrews-rhoades-zwegers}
G.~E.~Andrews, R.~C.~Rhoades and S.~P.~Zwegers, \emph{Modularity of the concave composition generating function}, Algebra Number Theory~\textbf{7} no. 9 (2013), 2103--2139.

\bibitem{bandix1}
K.~Banerjee and A.~Dixit, \emph{New representations for $\sigma(q)$ via reciprocity theorems}, Analytic Number Theory, Modular Forms and $q$-Hypergeometric Series (in honor of Krishnaswami Alladi's $60^{\text{th}}$ birthday), Springer Proceedings in Mathematics and Statistics, 2017, pp.~39--57.
%\bibitem{chan}
%S.~H.~Chan, \emph{Generalized Lambert series identities}, Proc. London Math. Soc. (3) \textbf{91} (2005), 598--622.

\bibitem{ntsr}
B.~C.~Berndt, \emph{Number Theory in the Spirit of Ramanujan}, Amer. Math. Soc., Providence, RI, 2006.

\bibitem{chan sot}
S.~H.~Chan, \emph{On Sears’s general transformation formula for basic hypergeometric series}, Ramanujan J.~\textbf{20} (2009), 69--79.
 
 \bibitem{cohen}
 H.~Cohen, \emph{$q$-Identities for Maass waveforms}, Invent. Math.~\textbf{91} (1988), 409--422.
 
 \bibitem{coogan}
 G.~G.~H.~Coogan, \emph{More generating functions for $L$-function values}, $q$-Series with Applications to Combinatorics, Number Theory, and Physics (Urbana, IL, 2000), Contemp. Math., 291, Amer. Math. Soc., Providence, RI, 2001, 109--114.
 
 \bibitem{coogan-ono}
 G.~G.~H.~Coogan and K.~Ono, \emph{A $q$-series identity and the arithmetic of Hurwitz zeta functions}, Proc.~Amer.~Math.~Soc.~\textbf{131} No. 3 (2003), 719--724.
 
 \bibitem{dks2_rascoe}
 A.~Dixit, G.~Kumar and A.~Srivastava, \emph{Non-Rascoe partitions and a rank parity function associated to the Rogers-Ramanujan partitions}, submitted for publication.
 
% \bibitem{fine}
% N.~J. Fine, \emph{Basic Hypergeometric Series and Applications}, Amer. Math. Soc. Providence, 1988.
 
\bibitem{folsom sot}
A.~Folsom, E.~Pratt, N.~Solomon and A.~R.~Tawfeek, \emph{Quantum Jacobi forms and sums of tails identities}, Res. Number Theory~ 8:8 (2022) (24 pages).

\bibitem{merca}
M.~Merca, \emph{Combinatorial interpretations of a recent
	convolution for the number of divisors of a positive integer}, J.~Number Theory~\textbf{160} (2016), 60--75.

\bibitem{gupta}
R.~Gupta, \emph{On sums-of-tails identities}, J.~Combin.~Theory~Ser.~A~\textbf{184} (2021), 105521 (28 pages).
	
				\bibitem{lnb}
	S.~Ramanujan, \emph{The Lost Notebook and Other Unpublished
		Papers}, Narosa, New Delhi, 1988.
	
\bibitem{Wptf}
S.~O.~Warnaar, \emph{Partial Theta functions. I. Beyond the lost notebook}, Proc. London Math. Soc.~\textbf{87} no. 2 (2003), 363--395.

\bibitem{zagiertop}
D.~Zagier, \emph{Vassiliev invariants and a strange identity related to the Dedekind eta-function}, Topology~\textbf{40} (2001), 945--960.

\bibitem{zagierqmf}
D.~Zagier, \emph{Quantum modular forms}, Quantas of Maths, Clay Math. Proc., Vol. 11, Amer. Math. Soc., Providence, RI, 2010, pp. 659--675.

\end{thebibliography}
\end{document}